\documentclass[11pt,reqno]{amsart}
\usepackage[hyphens]{url} 
\numberwithin{equation}{section}

\DeclareRobustCommand{\rchi}{{\mathpalette\irchi\relax}}
\newcommand{\irchi}[2]{\raisebox{\depth}{$#1\chi$}}

\makeatletter
\@namedef{subjclassname@2020}{%
  \textup{2020} Mathematics Subject Classification}
\makeatother

\usepackage{amsmath,amssymb,color}
\usepackage{amsthm} 
\usepackage{multicol}
\usepackage{mathtools}
\usepackage[bookmarksnumbered,colorlinks]{hyperref}
\usepackage{enumitem}
\mathtoolsset{showonlyrefs}
\usepackage[pagewise,mathlines]{lineno}
\usepackage[latin10, utf8]{inputenc}

\theoremstyle{plain} 
\newtheorem{theorem}{\indent\sc Theorem}[section]
\newtheorem{lemma}[theorem]{\indent\sc Lemma}
\newtheorem{corollary}[theorem]{\indent\sc Corollary}
\newtheorem{proposition}[theorem]{\indent\sc Proposition}

\theoremstyle{definition} 

\title[Variational problems of total mean curvature surfaces and applications]{\sc Total mean curvature surfaces in the product space $\mathbb{S}^{n}\times\mathbb{R}$ and applications}

\author[A.L. Albujer, S.F. da Silva and F.R. dos Santos]{Alma L. Albujer, Sylvia F. da Silva and F\'{a}bio R. dos Santos}

\address{
Departamento de Matem\'aticas \\
Universidad de C\'ordoba \\
14071 Campus de Rabanales, C\'ordoba\\
Spain
}
\email{alma.albujer@uco.es}

\address{
Departamento de matem\'atica\\
Universidade Federal de Pernambuco \\
50.740-560, Recife, Pernambuco \\
Brazil
}

\email{fabio.reis@ufpe.br}
\email{sylvia.ferreira@ufpe.br}


\thanks{$^{\ast}$Corresponding author}

\begin{document}

\maketitle


\begin{abstract}
The total mean curvature functional for submanifolds into the Riemannian product space $\mathbb{S}^n\times\mathbb{R}$ is considered and its first variational formula is presented. Later on, two second order differential operators are defined and a nice integral inequality relating both of them is proved. Finally we prove our main result: an integral inequality for closed stationary $\mathcal{H}$-surfaces in $\mathbb{S}^n\times\mathbb{R}$, characterizing the cases where the equality is attained.
\end{abstract}

\maketitle

\section{Introduction}

Along the last decades, integral inequalities have become an interesting tool for the study of rigidity results for closed submanifolds immersed in Riemannian spaces. In this setting, we point out that the first contribution in this thematic was given at 60's by Simons~\cite{Simons:68} who computed the Laplacian of the squared norm of the second fundamental form $\sigma$ of a minimal submanifold in the sphere. As a consequence, he showed that if $\Sigma^{m}$ is a closed minimal submanifold in $\mathbb{S}^{n}$, the following integral inequality holds:
\begin{equation}\label{Simons}
	\int_{\Sigma}|\sigma|^{2}\left(|\sigma|^{2}-c(n,m)\right)d\Sigma\geq0\quad\mbox{with}\quad c(n,m)=\frac{m(n-m)}{2(n-m)-1},
\end{equation}
where $d\Sigma$ is the volume element on $\Sigma^{m}$. Simons noticed that the inequality~\eqref{Simons} provides a natural gap concerning the size of the squared norm of the second fundamental form. Indeed, if the second fundamental form satisfies $0\leq|\sigma|^{2}\leq c(n,m)$ then either $|\sigma|^{2}=0$, and $\Sigma^{m}$ is totally geodesic, so a sphere $\mathbb{S}^m$, or $|\sigma|^{2}=c(n,m)$. This last equality was studied by Chern, do Carmo and Kobayashi~\cite{do Carmo:70}, who concluded that in this case $\Sigma^m$ is necessarily a Cliffod torus or a Veronese surface in $\mathbb{S}^4$. It is worth pointing out that the case of codimension 1 was also studied simultaneous and independently by Lawson~\cite{Lawson:69}. Nowadays, the inequality~\eqref{Simons} is known as the Simons integral inequality.

On the other hand, an interesting line of research is to study which submanifolds are critical points of certain geometric functionals. In this scenario, let us highlight three classical different functionals. Firstly, Chen considered in~\cite{Chen:74} the following functional for closed surfaces $\Sigma^2$ in $\mathbb{R}^3$:
\begin{equation}\label{Willmore-tilde}
	\widetilde{\mathcal{W}}(\Sigma)=\dfrac{1}{2}\int_{\Sigma}|\phi|^{2}d\Sigma=\int_{\Sigma} (H^2-K) d\Sigma,
\end{equation}
where $\phi=A-H I$ is the umbilicity tensor of $\Sigma$, $A$ denotes the shape operator of $\Sigma$ and $H$ and $K$ stand for the mean and Gaussian curvature of $\Sigma$, respectively. Closely related to~\eqref{Willmore-tilde} we can consider the well-known Willmore energy or Willmore functional given by:
\begin{equation}\label{Willmore}
	\mathcal{W}(\Sigma)=\int_{\Sigma} H^2 d\Sigma.
\end{equation}
In fact, because of the classical Gauss-Bonnet theorem, both functionals $\widetilde{\mathcal{W}}$ and $\mathcal{W}$ have the same critical points in the set of closed surfaces in $\mathbb{R}^3$. Associated to~\eqref{Willmore}, there is the famous Willmore conjecture, proposed in 1965 by Willmore~\cite{Willmore:65} and solved in $2014$ by Marques and Neves~\cite{Marques:14}, which guarantees that the value of $\mathcal{W}(\Sigma)$ is at least $2\pi^{2}$ when $\Sigma^{2}$ is an immersed torus into $\mathbb{R}^{3}$. 

Finally,  another interesting functional, the {\em total mean curvature} functional, was introduced by Chen~\cite{Chen:71} for any closed submanifold $\Sigma^{m}$ in the Euclidean space $\mathbb{R}^n$:
\begin{equation}\label{total}
	\mathcal{H}(\Sigma)=\int_{\Sigma} H^{m}d\Sigma.
\end{equation}
Chen proved that $\mathcal{H}$ is bounded from below by the volume of the unit $m$-sphere, being the equality attained precisely when the submanifold is the unit $m$-sphere. The total mean curvature functional has also been considered for submanifolds in other ambient spaces. In the case of closed submanifolds in the sphere $\mathbb{S}^n$, $\mathcal{H}$ is bounded from below by zero and the equality is attached at all closed minimal submanifolds of $\mathbb{S}^n$. Considering the variational problem associated to such functional, it is said that a submanifold $\Sigma^{m}$ is an {\em$\mathcal{H}$-submanifold} if it is a stationary point for the functional $\mathcal{H}$. In this context, Guo and Yin~\cite{Guo:16} established an integral inequality relating the total umbilicity tensor and the Euler characteristic $\chi(\Sigma)$ of a closed $\mathcal{H}$-surface $\Sigma^{2}$ immersed in $\mathbb{S}^{n}$:
\begin{equation}\label{eq:Guorelation}
	\int_{\Sigma}\left\{|\phi|^{2}\left(1-\left(2-\frac{1}{n-2}\right)|\phi|^{2}\right)+2\right\}d\Sigma\leq4\pi\rchi(\Sigma),
\end{equation}
being the equality achieved if and only if $\Sigma^{2}$ is either a totally geodesic $2$-sphere, a Clifford torus in $\mathbb{S}^{3}$ or a Veronese surface in $\mathbb{S}^{4}$.

Considering more general ambient spaces, recently the first and third authors computed in~\cite{Albujer:22} the Euler-Lagrange equation of a suitable Willmore functional for closed immersed surfaces in an homogeneous space $\mathbb{E}^{3}(\kappa,\tau)$. As an application, they developed a Simons type integral inequality for such surfaces, characterizing the surfaces for which the equality holds as Clifford or Hopf tori of the ambient space. Furthermore, recently the last two authors obtained an integral inequality for closed immersed submanifolds $\Sigma^{m}$ into the product space $\mathbb{S}^{n}\times\mathbb{R}$ having parallel normalized mean curvature vector field,~\cite{dos Santos:21.1,dos Santos:22}. They also showed that, in this case, the equality is attained if and only if $\Sigma^{m}$ is isometric to either a totally umbilical sphere, or to a certain family of Clifford tori in a totally geodesic sphere $\mathbb{S}^{m+1}$ of $\mathbb{S}^{n}$. 

In the spirit of the previous results, we will obtain the Euler-Lagrange equation of the total mean curvature functional for closed immersed surfaces into the product space $\mathbb{S}^{n}\times\mathbb{R}$, Proposition~\ref{prop:2}. As a consequence, we will get a Simons type integral inequality and we will characterize when the equality is attained. Specifically, if $\phi$ and $\phi_{h}$ stand for the umbilicity tensor of $\Sigma^m$, and the umbilicity tensor related to the mean curvature vector field $h$, respectively, and $T$ denotes the tangential part of the vector field $\partial_{t}$ in $\mathbb{S}^{n}\times\mathbb{R}$, the main aim of the paper is to prove the following result: 

\begin{theorem}\label{teo:1}
	Let $\Sigma^{2}$ be a closed immersed $\mathcal{H}$-surface into the product space $\mathbb{S}^{n}\times\mathbb{R}$. Then,
	\begin{equation}\label{eq:main}
		\int_{\Sigma}\left\{|\phi|^{2}\left(1-5|T|^{2}-\dfrac{3}{2}|\phi|^{2}\right)-2(|\phi_{h}|+1)|T|^{2}+2\right\}d\Sigma\leq4\pi\rchi(\Sigma).
	\end{equation}
	In particular, the equality holds if and only if $\Sigma^{2}$ is isometric to either
	\begin{itemize}
		\item[(i)] a slice $\mathbb{S}^{2}\times\{t_{0}\}$, or
		\item[(ii)] a totally geodesic $2$-sphere or a Clifford torus in $\mathbb{S}^{3}\times\{t_{0}\}$, or
		\item[(iii)] a Veronese surface in $\mathbb{S}^{4}\times\{t_{0}\}$,
	\end{itemize}
	for some $t_{0}\in\mathbb{R}$.
\end{theorem}

On the one hand, let us remark that, since given $m,n\in\mathbb{N}$, $m<n$, the unit sphere $\mathbb{S}^m$ is a totally geodesic submanifold of the unit sphere $\mathbb{S}^n$, the above surfaces for which the equality in~\eqref{eq:main} is attained are in fact surfaces of the product $\mathbb{S}^n\times\mathbb{R}$ in general dimension. On the other hand, let us also observe that~\eqref{eq:main} do not depend on the codimension. Besides that, in the case where $\Sigma^2$ is contained in a slice of $\mathbb{S}^n\times\mathbb{R}$, $T=0$. Thus,~\eqref{eq:main} reduces to
\begin{equation}\label{eq:reduced}
	\int_{\Sigma}\left\{|\phi|^{2}\left(1-\dfrac{3}{2}|\phi|^{2}\right)+2\right\}d\Sigma\leq4\pi\rchi(\Sigma),
\end{equation}
which in the case $n=4$ coincides with Guo and Yin's inequality,\eqref{eq:Guorelation}, and it improves it when $n>4$.

\section{Preliminaries}

In this section, we will present some basic facts about the product manifold $\mathbb{S}^n\times\mathbb{R}$, as well as a suitable Simons type formula for submanifolds immersed in such product. 

As usual, let $\mathbb{R}^{n+2}$ be the $(n+2)$-dimensional Euclidean space. Then, the product space $\mathbb{S}^{n}\times\mathbb{R}$ is defined as the following subset of $\mathbb{R}^{n+2}$:
\begin{equation}
	\mathbb{S}^{n}\times\mathbb{R}=\{(x_{1},\ldots,x_{n+2})\in\mathbb{R}^{n+2}\,;\,x_{1}^{2}+\cdots+x_{n+1}^{2}=1\},
\end{equation}
equipped with the induced metric from the Euclidean space, $\langle,\rangle$, i.e. $\mathbb{S}^{n}\times\mathbb{R}$ is the usual product of the unit sphere $\mathbb{S}^{n}(1)$ and the real line. Associated to it,
\begin{equation}\label{eq:1}
	\partial_{t}:=(\partial/\partial_{t})\big|_{(p,t)},\quad(p,t)\in \mathbb{S}^{n}\times\mathbb{R},
\end{equation}
is a parallel and unitary vector field, that is,
\begin{equation}\label{eq:2}
	\overline{\nabla}\partial_{t}=0\quad\mbox{and}\quad\langle\partial_{t},\partial_{t}\rangle=1,
\end{equation}
where $\overline{\nabla}$ is the Levi-Civita connection of $\mathbb{S}^{n}\times\mathbb{R}$. 

Concerning the curvature tensor of $\mathbb{S}^{n}\times\mathbb{R}$, it is well-known that it satisfies (see~\cite{Daniel:07}):
\begin{equation}\label{eq:5}
	\begin{split}
		\overline{R}(X,Y)Z=&\,\,\langle X,Z\rangle Y-\langle Y,Z\rangle X+\langle Z,\partial_{t}\rangle(\langle Y,\partial_{t}\rangle X-\langle X,\partial_{t}\rangle Y)\\
		&+\,(\langle Y,Z\rangle\langle X,\partial_{t}\rangle-\langle X,Z\rangle\langle Y,\partial_{t}\rangle)\partial_{t},
	\end{split}
\end{equation}
where $X,Y,Z\in\mathfrak{X}(\mathbb{S}^n\times\mathbb{R})$ and $\overline{R}$ is defined by (see~\cite{O'Neill:83})
\begin{equation}
	\overline{R}(X,Y)Z=\overline{\nabla}_{[X,Y]}Z-[\overline{\nabla}_{X},\overline{\nabla}_{Y}]Z.
\end{equation}

Let us consider $\Sigma^{m}$ an $m$-dimensional submanifold of $\mathbb{S}^{n}\times\mathbb{R}$ and let us also denote by $\langle,\rangle$ the induced metric on $\Sigma^m$. In this setting, we will denote by $\nabla$ the Levi-Civita connection of $\Sigma^{m}$ and $\nabla^{\perp}$ will stand for the normal connection of $\Sigma^{m}$ in $\mathbb{S}^{n}\times\mathbb{R}$. We will denote by $\sigma$ the second fundamental form of $\Sigma^{m}$ in $\mathbb{S}^{n}\times\mathbb{R}$ and by $A_{\xi}$ the Weingarten operator associated to a fixed normal vector field $\xi\in\mathfrak{X}(\Sigma)^{\perp}$. We note that, for each $\xi\in\mathfrak{X}(\Sigma)^{\perp}$, $A_{\xi}$ is a symmetric endomorphism of the tangent space $T_p\Sigma$ at $p\in\Sigma^{m}$. Moreover, $A_{\xi}$ and $\sigma$ are related by
\begin{equation}
	\langle\sigma(X,Y),\xi\rangle=\langle A_{\xi}(X),Y\rangle,
\end{equation}
for all $X,Y\in\mathfrak{X}(\Sigma)$ and $\xi\in\mathfrak{X}(\Sigma)^{\perp}$. We also recall that the Gauss and Weingarten formulas of $\Sigma^{m}$ in $\mathbb{S}^{n}\times\mathbb{R}$ are given by
\begin{equation}\label{Gauss and Weingarten}
	\overline{\nabla}_{X}Y=\nabla_{X}Y+\sigma(X,Y)\quad\mbox{and}\quad\overline{\nabla}_{X}\xi=-A_{\xi}(X)+\nabla^{\perp}_{X}\xi,
\end{equation}
for all $X,Y\in\mathfrak{X}(\Sigma)$ and $\xi\in\mathfrak{X}(\Sigma)^{\perp}$.

Since $\partial_t\in\mathfrak{X}(\mathbb{S}^n\times\mathbb{R})$, it can be decomposed along $\Sigma^m$ as
\begin{equation}\label{eq:3}
	\partial_{t}=T+N
\end{equation}
where $T:=\partial_{t}^{\top}$ and $N:=\partial_{t}^{\perp}$ denote, respectively, the tangent and normal part of the vector field $\partial_{t}$ on the tangent and normal bundle of the submanifold $ \Sigma^{m}$ in $\mathbb{S}^{n}\times\mathbb{R}$. Moreover, from~\eqref{eq:2} and~\eqref{eq:3}, we get the relation
\begin{equation}\label{eq:4}
	1=\langle\partial_{t},\partial_{t}\rangle=|T|^{2}+|N|^{2},
\end{equation}
$|\cdot|$ being the norm related to the metric $\langle,\rangle$. It is clear that, if $T$ vanishes identically along $\Sigma$, then $\partial_{t}$ is normal to $\Sigma^{m}$ and hence $\Sigma^{m}$ lies in a slice $\mathbb{S}^{n}\times\{t_0\}$, $t_0\in\mathbb{R}$. Besides that, a direct computation from~\eqref{eq:2} and~\eqref{Gauss and Weingarten} gives
\begin{equation}\label{eq:6}
	\nabla_{X}T=A_{N}(X)\quad\mbox{and}\quad\nabla_{X}^{\perp}N=-\sigma(T,X),\quad\mbox{for all}\quad X\in\mathfrak{X}(\Sigma).
\end{equation}

A well-known fact is that the curvature tensor $R$ of $\Sigma^{m}$ can be described in terms of its second fundamental form $\sigma$ and the curvature tensor $\overline{R}$ of the ambient space $\mathbb{S}^{n}\times\mathbb{R}$ by the so-called {\em Gauss equation}, which is given by
\begin{equation}\label{eq:Gauss}
	\begin{split}
		\langle R(X,Y)Z,W\rangle&=\langle\overline{R}(X,Y)Z,W\rangle+\langle\sigma(X,Z),\sigma(Y,W)\rangle-\langle\sigma(Y,Z),\sigma(X,W)\rangle\\
		&=\langle X,Z\rangle\langle Y,W\rangle-\langle Y,Z\rangle\langle X,W\rangle+\langle Z,T\rangle(\langle Y,T\rangle\langle X,W\rangle-\langle X,T\rangle\langle Y,W\rangle)\\
		&\quad+(\langle Y,Z\rangle\langle X,T\rangle-\langle X,Z\rangle\langle Y,T\rangle)\langle T,W\rangle\\
		&\quad\,\,\,+\langle\sigma(X,Z),\sigma(Y,W)\rangle-\langle\sigma(Y,Z),\sigma(X,W)\rangle,
	\end{split}
\end{equation}
for all $X,Y,Z,W\in\mathfrak{X}(\Sigma)$, and the {\em Codazzi equation}
\begin{equation}\label{eq:Codazzi}
	(\nabla^{\perp}_{Y}\sigma)(X,Z)-(\nabla^{\perp}_{X}\sigma)(Y,Z)=(\overline{R}(X,Y)Z)^{\perp}=\left(\langle Y,Z\rangle\langle X,T\rangle-\langle X,Z\rangle\langle Y,T\rangle\right)N,
\end{equation}
for all $X,Y,Z\in\mathfrak{X}(\Sigma)$, where $\nabla^{\perp}\sigma$ satisfies:
\begin{equation}\label{derivative condition}
	(\nabla^{\perp}_{X}\sigma)(Y,Z)=\nabla^{\perp}_{X}\sigma(Y,Z)-\sigma(\nabla_{X}Y,Z)-\sigma(Y,\nabla_{X}Z).
\end{equation}

Let us denote by $h$ the mean curvature vector field of $\Sigma^m$ in $\mathbb{S}^{n}\times\mathbb{R}$, defined by
\begin{equation}\label{mean curvature}
	h=\dfrac{1}{m}{\rm tr}(\sigma)
\end{equation}
and by $H$ its norm, i.e. $H^2=\langle h,h \rangle$. It is immediate to check that if $\{e_{m+1},\ldots,e_{n+1}\}$ is an orthonormal frame of $\mathfrak{X}(\Sigma)^\perp$, we can write~\eqref{mean curvature} in the following way:
\begin{equation}\label{mean curvature2}
	h=\sum_{\alpha}H^{\alpha}e_{\alpha}\quad\mbox{where}\quad H^{\alpha}:=\frac{1}{m}{\rm tr}(A_{\alpha})=\langle h,e_\alpha\rangle,
\end{equation}
and $A_\alpha:=A_{e_\alpha}$. In particular, $mH^{2}={\rm tr}(A_{h})$.

Next, for any normal vector field $\xi$, let us define the tensor $\phi_\xi$ as the traceless part of $A_{\xi}$, i.e. $\phi_{\xi}:=A_{\xi}-\frac{1}{m}{\rm tr}(A_{\xi})I$. We shall also consider $\phi$ the traceless part of $\sigma$, given by
\begin{equation}\label{traceless-tensor}
	\phi(X,Y):=\sigma(X,Y)-\langle X,Y\rangle h.
\end{equation}
The tensors $\phi$ and $\phi_\xi$ are also known as the umbilicity tensor and the umbilicity tensor related to $\xi$ of $\Sigma^m$, respectively. It is easy to check that
\begin{equation}\label{eq:7}
	|\phi|^{2}=|\sigma|^{2}-mH^{2}\quad\mbox{and}\quad|\phi_{\xi}|^{2}=|A_{\xi}|^{2}-m\langle\xi,h\rangle^{2}.
\end{equation}
Observe that $|\phi|^{2}=0$ if and only if $ \Sigma^{m}$ is a totally umbilical submanifold of $\mathbb{S}^{n}\times\mathbb{R}$.

We end this section by recalling the following two results, which we shall use later in this paper. The first one is a Simons type formula proved in~\cite{dos Santos:21,dos Santos:21.1}. It should be noticed that, for the sake of simplicity, in Proposition~\ref{prop:1} and, in general, along this manuscript we will naturally identify, at convenience, the Weingarten operator with its associated symmetric matrix.
\begin{proposition}\label{prop:1}
	Let $\Sigma^{m}$ be a submanifold in the product space $\mathbb{S}^{n}\times\mathbb{R}$. Then
	\begin{equation}
		\begin{split}
			\dfrac{1}{2}\Delta|\sigma|^{2}&=|\nabla^{\perp}\sigma|^{2}+m\sum_{\alpha}{\rm tr}(A_{\alpha}\circ{\rm Hess}\,H^{\alpha})+m|\phi_{N}|^{2}-2m\sum_{\alpha}|\phi_{\alpha}(T)|^{2}+(m-|T|^{2})|\phi|^{2}\\
			&\quad-m\langle\phi_{h}(T),T\rangle+\sum_{\alpha,\beta}{\rm tr}(A_{\beta}){\rm tr}(A^{2}_{\alpha} A_{\beta})-\sum_{\alpha,\beta}\left(N(A_{\alpha} A_{\beta}-A_{\beta} A_{\alpha})+[{\rm tr}(A_{\alpha} A_{\beta})]^{2}\right),
		\end{split}
	\end{equation}
	where $\phi_\alpha:=\phi_{e_\alpha}$, $m+1\leq \alpha,\beta \leq n+1$ and $N(B):={\rm tr}(BB^{t})$ for all matrix $B$.
\end{proposition}

The second one is an algebraic lemma which was proved in~\cite{Li:92}.
\begin{lemma}\label{lem:3.4}
	Let $B_1,\ldots,B_p$, where $p\geq2$, be symmetric $m\times m$ matrices. Then
	\begin{equation}
		\sum_{\alpha,\beta=1}^{p}\left(N(B_{\alpha}B_{\beta}-B_{\beta}B_{\alpha})+[{\rm tr}(B_\alpha B_\beta)]^{2}\right)\leq\frac{3}{2}\left(\sum_{\alpha=1}^{p}N(B_{\alpha})\right)^{2}.
	\end{equation}
\end{lemma}

\section{The first variation of the total mean curvature}

The goal of this section is to study the stationary points of the functional $\mathcal{H}$, defined in~\eqref{total}, for closed surfaces in $\mathbb{S}^n\times\mathbb{R}$. To that end, we will recall the rough Laplacian $\Delta^{\perp}:\mathfrak{X}(\Sigma)^{\perp}\to\mathfrak{X}(\Sigma)^{\perp}$ which is defined by setting
\begin{equation}\label{laplacian}
	\Delta^{\perp}\xi:={\rm tr}(\nabla^{2}\xi)=\sum_{i}\nabla^{\perp}_{e_{i}}\nabla^{\perp}_{e_{i}}\xi,
\end{equation}
where $\{e_{1},\ldots,e_{m}\}$ is any orthonormal frame of $\mathfrak{X}(\Sigma)$.

Now, let us compute the first variational formula of $\mathcal{H}$.

\begin{proposition}\label{prop:2}
	Let $x:\Sigma^{m}\rightarrow \mathbb{S}^{n}\times\mathbb{R}$ be an isometrically immersed closed submanifold. Then $x$ is a stationary point of $\mathcal{H}$, or an $\mathcal{H}$-submanifold, if and only if
	\begin{equation}\label{eq:8}
		H^{m-2}\left(\Delta^{\perp}h+\left(m-|T|^{2}-mH^{2}\right)h-m\langle N,h\rangle N+\sum_{\alpha,\beta}H^{\alpha}{\rm tr}(A_{\alpha} A_{\beta})e_{\beta}\right)=0,
	\end{equation}
	for $m>2$, and
	\begin{equation}\label{eq:8_2}
		\Delta^{\perp}h+\left(2-|T|^{2}-2H^{2}\right)h-2\langle N,h\rangle N+\sum_{\alpha,\beta}H^{\alpha}{\rm tr}(A_{\alpha} A_{\beta})e_{\beta}=0,
	\end{equation}
	in the case $m=2$, where $m+1\leq\alpha,\beta\leq n+1$.
\end{proposition}

\begin{proof}
	Let us consider a variation of $x$, that is, a smooth map $X:\Sigma^{m}\times(-\varepsilon,\varepsilon)\rightarrow \mathbb{S}^{n}\times\mathbb{R}$ satisfying {that for each} $s\in(-\varepsilon,\varepsilon)$, the map $X_{s}:\Sigma^{m}\rightarrow \mathbb{S}^{n}\times\mathbb{R}$,
	given by $X_{s}(p)=X(p,s)$, is an immersion {and} $X_{0}=x$. Then, we can compute the first variation of $\mathcal{H}$ along $X$, that is,
	\begin{equation}\label{eq:9}
		\dfrac{d}{ds}\mathcal{H}(X_{s})\bigg|_{s=0}=\int_{\Sigma}\dfrac{d}{ds}\left(H^{m}_{s}d\Sigma_{s}\right)\bigg|_{s=0},
	\end{equation}
	where, for each $s\in (-\varepsilon,\varepsilon)$, $H_s=\sqrt{\langle h_s,h_s\rangle}$ stands for the norm of the mean curvature vector of $\Sigma^{m}$ in $\mathbb{S}^{n}\times\mathbb{R}$ with respect to the metric induced by $X_s$ and $d\Sigma_{s}$ denotes its volume element.
	
	On the one hand, let us compute $\dfrac{d}{ds}H^m_s\bigg|_{s=0}$. For the sake of simplicity, let us denote $v=d/ds$. We claim that:
	\begin{equation}\label{eq:21_claim}
		\begin{split}
			\frac{m}{2}v(H_s^{2})\bigg|_{s=0}&=\langle mh-|T|^{2}h-m\langle N,h\rangle N+\sum_{\alpha,\beta}H^{\alpha}{\rm tr}(A_{\alpha} A_{\beta}) e_{\beta},v^{\perp}\rangle\\
			&\quad+\frac{m}{2}v^{\top}(H^{2})+\langle h,\Delta^{\perp}v^{\perp}\rangle.
		\end{split}
	\end{equation}
	
	Let us assume now that $m>2$. Then,
	\begin{equation}\label{eq:ddsHm_aux}
		v(H_s^m)=v((H_s^2)^{\frac{m}{2}})=\frac{m}{2}H_s^{m-2}v(H^2_s)
	\end{equation}
	and, consequently,
	\begin{equation}\label{eq:ddsHm}
		\begin{split}
			v(H_s^m)\bigg|_{s=0}&=H^{m-2}\langle mh-|T|^{2}h-m\langle N,h\rangle N+\sum_{\alpha,\beta}H^{\alpha}{\rm tr}(A_{\alpha} A_{\beta})e_\beta,v^{\perp}\rangle\\
			&\quad+ H^{m-2}\left(\frac{m}{2}v^{\top}(H^{2})+\langle h,\Delta^{\perp}v^{\perp}\rangle \right).
		\end{split}
	\end{equation}
	
	Furthermore, by using~\cite[Lemma $5.4$]{Weiner:78} (see also~\cite[Lemma 4.2]{Colares:97}), we have
	\begin{equation}\label{eq:ddsSs}
		v(d\Sigma_{s})\bigg|_{s=0}=\left(-m\langle h,v^{\perp}\rangle+{\rm div}(v^{\top})\right)d\Sigma.
	\end{equation}
	Therefore, along $\Sigma^m$, $m>2$, it holds
	\begin{equation}\label{eq:22}
		\begin{split}
			v(H_s^{m}d\Sigma_s)\bigg|_{s=0}&=v(H_s^{m})\bigg|_{s=0}d\Sigma+H^{m}v(d\Sigma_s)\bigg|_{s=0}\\
			&=\left\{H^{m-2}\left(\langle mh-|T|^2h-m\langle N,h\rangle N-mH^2h,v^\perp \rangle\right)\right\}d\Sigma\\
			&\quad+\left\{H^{m-2}\left( \sum_{\alpha,\beta} H^\alpha {\rm tr}(A_\alpha A_\beta)\langle e_\beta ,v^\perp \rangle+\langle h,\Delta^{\perp}v^{\perp}\rangle+{\rm div}(H^{m}v^{\top})\right)\right\}d\Sigma,
		\end{split}
	\end{equation}
	where it has been used~\eqref{eq:ddsHm},~\eqref{eq:ddsSs} and the fact that
	\begin{equation}\label{eq:23}
		{\rm div}(H^{m}v^{\top})=\frac{m}{2}H^{m-2}v^{\top}(H^{2})+H^{m}{\rm div}(v^{\top}).
	\end{equation}
	Consequently,
	\begin{equation}\label{eq:24}
		\begin{split}
			\dfrac{d}{ds}\int_{\Sigma}H_{s}^{m}d\Sigma_{s}\bigg|_{s=0}&=\int_{\Sigma}H^{m-2}\langle\Delta^{\perp}v^{\perp},h\rangle d\Sigma+\int_{\Sigma}H^{m-2}\sum_{\alpha,\beta}H^{\alpha}{\rm tr}(A_{\alpha} A_{\beta})\langle e_{\beta},v^{\perp}\rangle d\Sigma\\
			&\quad-\int_{\Sigma}H^{m-2}\left(\langle |T|^{2}h-mh+m\langle N,h\rangle N+mH^{2}h,v^{\perp}\rangle\right) d\Sigma.
		\end{split}
	\end{equation}
	Hence, $x$ is a stationary point of $\mathcal{H}$ if and only if
	\begin{equation}\label{eq:25}
		H^{m-2}\left(\Delta^{\perp}h-|T|^{2}h+mh-m\langle N,h\rangle N-mH^{2}h+\sum_{\alpha,\beta}H^{\alpha}{\rm tr}(A_{\alpha} A_{\beta})e_{\beta}\right)=0.
	\end{equation}
	The case $m=2$ follows with an analogous argument using~\eqref{eq:21_claim} instead of~\eqref{eq:ddsHm}.
	
	It remains to prove the claim. By~\eqref{mean curvature},
	\begin{equation}\label{eq:claim}
		mv(H_s^2)=\sum_i\langle (\overline{\nabla}_vA_{h_s})e_i,e_i\rangle=\sum_{i}\langle\overline{\nabla}_{v}A_{h_s}(e_{i}),e_{i}\rangle-\sum_{i}\langle A_{h_s}(\overline{\nabla}_{v}e_{i})^{\top},e_{i}\rangle,
	\end{equation}
	for any $\{e_1,\ldots,e_m\}$ orthonormal frame of $\mathfrak{X}(\Sigma)$. In particular, given $p\in\Sigma$ we can choose locally a totally geodesic frame, i.e. $(\nabla_{e_i}e_j)(p)=0$ for all $1\leq i,j\leq m$.
	
	Although in the following we will work at $p$, by simplicity we will omit the point. Let us denote
	\begin{equation}\label{eq:10_claim}
		I=\sum_{i}\langle\overline{\nabla}_{v}A_{h_s}(e_{i}),e_{i}\rangle\quad\mbox{and}\quad II=\sum_{i}\langle A_{h_s}(\overline{\nabla}_{v}e_{i})^{\top},e_{i}\rangle
	\end{equation}
	and let us compute both terms separately. From~\eqref{Gauss and Weingarten} and the fact that $[v,e_{i}]=\overline{\nabla}_{v}e_{i}-\overline{\nabla}_{e_{i}}v=0$, we have
	\begin{equation}\label{eq:11_claim}
		\begin{split}
			I&=-\sum_{i}\langle\overline{\nabla}_{v}\overline{\nabla}_{e_{i}}h_s,e_{i}\rangle+\sum_{i}\langle\overline{\nabla}_{v}\nabla^{\perp}_{e_{i}}h_s,e_{i}\rangle\\
			&=\sum_{i}\langle\overline{R}(v,e_{i})h_s,e_{i}\rangle-\sum_{i}\langle\overline{\nabla}_{e_{i}}\overline{\nabla}_{v}h_s,e_{i}\rangle-\sum_{i}\langle\nabla^{\perp}_{e_{i}}h_s,\overline{\nabla}_{v}e_{i}\rangle\\
			&=\sum_{i}\langle\overline{R}(v,e_{i})h_s,e_{i}\rangle-\sum_{i}e_{i}\langle\overline{\nabla}_{v}h_s,e_{i}\rangle+\sum_{i}\langle\overline{\nabla}_{v}h_s,\overline{\nabla}_{e_{i}}e_{i}\rangle-\sum_{i}\langle\nabla^{\perp}_{e_{i}}h_s,\overline{\nabla}_{e_{i}}v\rangle\\
			&=\sum_{i}\langle\overline{R}(v,e_{i})h_s,e_{i}\rangle+\sum_{i}e_{i}\langle h_s,\overline{\nabla}_{v}e_{i}\rangle+\sum_{i}\langle\overline{\nabla}_{v}h_s,\sigma(e_{i},e_{i})\rangle-\sum_{i}\langle\nabla^{\perp}_{e_{i}}h_s,\sigma(e_{i},v^{\top})\rangle\\
			&\quad-\sum_{i}\langle\nabla^{\perp}_{e_{i}}h_s,\nabla^{\perp}_{e_{i}}v^{\perp}\rangle\\
			&=\sum_{i}\langle\overline{R}(v,e_{i})h_s,e_{i}\rangle+\sum_{i}e_{i}\langle h_s,\sigma(e_{i},v^{\top})\rangle+\sum_{i}e_{i}\langle h_s,\nabla^{\perp}_{e_{i}}v^{\perp}\rangle+\dfrac{m}{2}v(H^{2}_s)\\
			&\quad-\sum_{i}\langle\nabla^{\perp}_{e_{i}}h_s,\sigma(v^{\top},e_{i})\rangle-\sum_{i}\langle\nabla^{\perp}_{e_{i}}h_s,\nabla^{\perp}_{e_{i}}v^{\perp}\rangle\\
			&=\sum_{i}\langle\overline{R}(v,e_{i})h_s,e_{i}\rangle+\sum_{i}\langle h_s,\nabla^{\perp}_{e_{i}}\sigma(v^{\top},e_{i})\rangle+\langle h_s,\Delta^{\perp}v^{\perp}\rangle+\dfrac{m}{2}v(H^{2}_s).
		\end{split}
	\end{equation}
	where we have also used~\eqref{mean curvature} and~\eqref{laplacian}.
	
	From the Codazzi equation~\eqref{eq:Codazzi} it holds 
	\begin{equation}\label{eq:13_claim}
		\begin{split}
			\nabla^{\perp}_{e_{i}}\sigma(v^{\top},e_{i})&=(\nabla^{\perp}_{e_{i}}\sigma)(v^{\top},e_{i})+\sigma(e_{i},\nabla_{e_{i}}v^{\top})\\
			&=(\overline{R}(v^{\top},e_{i})e_{i})^{\perp}+(\nabla^{\perp}_{v^{\top}}\sigma)(e_{i},e_{i})+\sigma(e_{i},\nabla_{e_{i}}v^{\top}).
		\end{split}
	\end{equation}
	Inserting~\eqref{eq:13_claim} in~\eqref{eq:11_claim},
	\begin{equation}\label{eq:14_claim}
		\begin{split}
			I&=\sum_{i}\langle\overline{R}(v^{\perp},e_{i})h_s,e_{i}\rangle+\sum_{i}\langle h_s,\nabla^{\perp}_{v^{\top}}\sigma(e_{i},e_{i})+\sigma(e_{i},\nabla_{e_{i}}v^{\top})\rangle+\langle h_s,\Delta^{\perp}v^{\perp}\rangle+\dfrac{m}{2}v(H_s^{2})\\
			&=\sum_{i}\langle\overline{R}(v^{\perp},e_{i})h_s,e_{i}\rangle+\dfrac{m}{2}v^{\top}(H_s^{2})+\sum_{i}\langle A_{h_s}(e_{i}),\nabla_{e_{i}}v^{\top}\rangle+\langle h_s,\Delta^{\perp}v^{\perp}\rangle+\dfrac{m}{2}v(H_s^{2}).
		\end{split}
	\end{equation}
	For the second expression of~\eqref{eq:10_claim}, it is not difficult to check that
	\begin{equation}\label{eq:15_claim}
		\begin{split}
			II&=\sum_{i}\langle A_{h_s}(\overline{\nabla}_{v}e_{i})^\top,e_{i}\rangle\\
			&=\sum_{i}\langle A_{h_s}(\overline{\nabla}_{e_{i}}v^{\top}+\overline{\nabla}_{e_{i}}v^{\perp}),e_{i}\rangle=\sum_{i}\langle A_{h_s}(\nabla_{e_{i}}v^{\top}),e_{i}\rangle-{\rm tr}(A_{h_s} A_{v^{\perp}}).
		\end{split}
	\end{equation}
	Therefore,
	\begin{equation}\label{eq:16_claim}
		\frac{m}{2}v(H_s^{2})=\sum_{i}\langle\overline{R}(v^{\perp},e_{i})h_s,e_{i}\rangle+\dfrac{m}{2}v^{\top}(H_s^{2})+\langle h_s,\Delta^{\perp}v^{\perp}\rangle+{\rm tr}(A_{h_s} A_{v^{\perp}}),
	\end{equation}
	thus
	\begin{equation}\label{eq:16_bis_claim}
		\frac{m}{2}v(H_s^{2})\bigg|_{s=0}=\sum_{i}\langle\overline{R}(v^{\perp},e_{i})h,e_{i}\rangle+\dfrac{m}{2}v^{\top}(H^{2})+\langle h,\Delta^{\perp}v^{\perp}\rangle+{\rm tr}(A_{h}A_{v^{\perp}}),
	\end{equation}
	On the other hand, writing $v^{\perp}=\sum_{\beta}\langle v^{\perp},e_{\beta}\rangle e_{\beta}$ and $h=\sum_{\alpha} H^\alpha e_\alpha$, from~\eqref{mean curvature} we get 
	\begin{equation}\label{eq:18_claim}
		A_{h}=\sum_{\alpha}H^{\alpha}A_{\alpha}\quad\mbox{and}\quad A_{v^{\perp}}=\sum_{\beta}\langle v^{\perp},e_{\beta}\rangle A_{\beta}.
	\end{equation}
	Hence,
	\begin{equation}\label{eq:19_claim}
		{\rm tr}(A_{h} A_{v^{\perp}})=\sum_{\alpha,\beta}H^{\alpha}{\rm tr}(A_{\alpha} A_{\beta})\langle v^{\perp},e_{\beta}\rangle.
	\end{equation}
	Besides this, from~\eqref{eq:5},
	\begin{equation}\label{eq:20_claim}
		\sum_{i=1}^{m}\langle\overline{R}(v^{\perp},e_{i})h,e_{i}\rangle=\langle mh-|T|^{2}h-m\langle N,h\rangle N,v^{\perp}\rangle.
	\end{equation}
	So, the claim is proved by replacing~\eqref{eq:19_claim} and~\eqref{eq:20_claim} into~\eqref{eq:16_bis_claim}.
\end{proof}

It is not difficult to see that minimal submanifolds are stationary points of the total mean curvature functional $\mathcal{H}$. In fact,~\eqref{eq:8} is trivial for minimal submanifolds and~\eqref{eq:8_2} is also satisfied since $H=0$ implies that the mean curvature vector field $h$ vanishes identically at $\Sigma^{m}$. Let us prove that minimal submanifolds are the only stationary points in the class of totally umbilical submanifolds contained in a slice of $\mathbb{S}^n\times\mathbb{R}$. To that end, we need to present first the following auxiliary result.

\begin{lemma}\label{lem:parallel_h}
	If $\Sigma^{m}$ is a totally umbilical submanifold contained in a slice of $\mathbb{S}^{n}\times\mathbb{R}$, then the mean curvature vector field is parallel in the normal bundle.
\end{lemma}

\begin{proof}
	From umbilicity, \eqref{traceless-tensor} gives
	\begin{equation}\label{eq:28}
		\sigma(X,Y)=\langle X,Y\rangle h,\qquad X,Y\in\mathfrak{X}(\Sigma).
	\end{equation}
	Hence, a direct computation from the Codazzi equation~\eqref{eq:Codazzi} yields
	\begin{equation}\label{eq:29}
		\left(\langle Y,Z\rangle\langle X,T\rangle-\langle X,Z\rangle\langle Y,T\rangle\right)N=\langle X,Z\rangle\nabla^{\perp}_{Y}h-\langle Y,Z\rangle\nabla^{\perp}_{X}h,
	\end{equation}
	for all $X,Y,Z\in\mathfrak{X}(\Sigma)$. Since $\Sigma^m$ is contained in a slice, $T=0$, so from~\eqref{eq:29}
	\begin{equation}\label{eq:30}
		\langle Y,Z\rangle\nabla^{\perp}_{X}h=\langle X,Z\rangle\nabla^{\perp}_{Y}h,
	\end{equation}
	for all $X,Y,Z\in\mathfrak{X}(\Sigma)$. Therefore, choosing $Y=Z$ orthogonal to $X$, we conclude that $h$ is parallel in the normal bundle.
\end{proof}

As a consequence of the previous result, we get the following corollary.

\begin{corollary}\label{cor:tu_tg}
	Let $\Sigma^m$ be a totally umbilical submanifold contained in a slice of $\mathbb{S}^{n}\times\mathbb{R}$. Then, $\Sigma^m$ is an $\mathcal{H}$-submanifold of $\mathbb{S}^n\times\mathbb{R}$ if and only if it is totally geodesic.
\end{corollary}
\begin{proof}
	Let $\Sigma^m$ be a submanifold of $\mathbb{S}^n\times\mathbb{R}$ under the assumptions of the corollary. From Lemma~\ref{lem:parallel_h} it follows that $\nabla^{\perp}h=0$. Furthermore, since $\Sigma^m$ is contained in a slice, $T=0$. Thus, using~\eqref{eq:6} and the assumption of umbilicity, we have
	\begin{equation}
		0= \langle A_{N}(X), Y\rangle=\langle \sigma(X,Y),N \rangle=\langle X,Y\rangle\langle h,N\rangle
	\end{equation}
	for all $X,Y\in\mathfrak{X}(\Sigma)$, so $\langle h,N\rangle=0$. Besides that, the umbilicity of $\Sigma^{m}$ also implies that, for every $m+1\leq\alpha\leq n+1$, it holds $A_{\alpha}(X)=\langle h,e_{\alpha}\rangle X$.
	
	Hence, the first variational formula for $\mathcal{H}$ in Proposition~\ref{prop:2} becomes
	\begin{equation}
		0=H^{m-2}\left(mh-mH^{2}h+m\sum_{\alpha} \langle h,e_{\alpha}\rangle^{2}h\right)=mH^{m-2}h
	\end{equation}
	if $m>2$ and simply $h=0$ in the case $m=2$. In any case, it is immediate to check that $\Sigma^{m}$ is an $\mathcal{H}$-submanifold if and only if it is minimal. Thus, from umbilicity, if and only if it is totally geodesic.
\end{proof}

\section{Two key lemmas}

Associated to the second fundamental form of
$\Sigma^{m}$, let us consider the following operator $P:\mathfrak{X}(\Sigma)\times\mathfrak{X}(\Sigma)\to\mathfrak{X}(\Sigma)^{\perp}$ by setting 
\begin{equation}\label{eq:P}
	P(X,Y)=m\langle X,Y\rangle h-\sigma(X,Y).
\end{equation}
We observe that $P$ is symmetric and ${\rm tr}(P)=m(m-1)h$. Concerning to $P$, let us consider the following second order differential operator:
\begin{equation}\label{eq:31}
	\square^{\ast}:\mathfrak{X}(\Sigma)^{\perp}\to\mathcal{C}^{\infty}(\Sigma)
\end{equation}
given by
\begin{equation}\label{eq:32}
	\square^{\ast}(\xi)=\langle P,\nabla^{2}\xi\rangle,
\end{equation}
where $\langle\,,\rangle$ denotes the Hilbert-Schmidt inner product. We observe that, for each $\alpha\in\{m+1,n+1\}$, by~\eqref{eq:P} it holds
\begin{equation}\label{eq:33}
	\begin{split}
		\langle P(X,Y),e_{\alpha}\rangle&=m\langle X,Y\rangle\langle h,e_{\alpha}\rangle-\langle\sigma(X,Y),e_{\alpha}\rangle\\
		&=m\langle X,Y\rangle H^{\alpha}-\langle A_{\alpha}(X),Y\rangle,
	\end{split}
\end{equation}
which motivates the definition of the operator $P_{\alpha}:\mathfrak{X}(\Sigma)\to\mathfrak{X}(\Sigma)$ given by $P_{\alpha}=mH^{\alpha}I-A_{\alpha}$. It is immediate to see that $P_{\alpha}$ is symmetric, ${\rm tr}(P_{\alpha})=m(m-1)H^{\alpha}$ and
\begin{equation}\label{eq:34}
	\sum_{\alpha}{\rm tr}(P_{\alpha})e_{\alpha}=m(m-1)\sum_{\alpha}H^{\alpha}e_{\alpha}=m(m-1)h={\rm tr}(P).
\end{equation}

We can also define another second differential operator
\begin{equation}\label{eq:35_2}
	\square:\mathcal{C}^{\infty}(\Sigma)\to\mathfrak{X}(\Sigma)^\perp
\end{equation}
such that
\begin{equation}\label{eq:35}
	\square(f)=\sum_{\alpha}{\rm tr}(P_{\alpha}\circ{\rm Hess}\,f)e_{\alpha}.
\end{equation}
The following result gives a relation between both operators $\square^{\ast}$ and $\square$.

\begin{lemma}\label{lem:1}
	Let $\Sigma^{m}$ be a closed submanifold in the product space $\mathbb{S}^{n}\times\mathbb{R}$. Then
	\begin{equation}
		\int_{\Sigma}f\,\square^{\ast}(\xi)d\Sigma=\int_{\Sigma}\langle\square(f),\xi\rangle d\Sigma+(m-1)\int_{\Sigma}\left(f\langle\nabla^{\perp}_{T}\xi,N\rangle-\langle N,\xi\rangle\langle\nabla f,T\rangle\right)d\Sigma,
	\end{equation}
	for all $f\in\mathcal{C}^{2}(\Sigma)$ and $\xi\in T\Sigma^{\perp}$.
\end{lemma}

\begin{proof}
	Let $p\in\Sigma^{m}$ and $\{e_1,\ldots,e_m\}$ be an orthonormal frame of $\mathfrak{X}(\Sigma)$ on a neighborhood $U\subset\Sigma^{m}$ of $p$, geodesic at $p$, that is, $(\nabla_{e_{i}}e_{j})(p)=0$ for all $1\leq i,j\leq m$. By using the Hilbert-Schmidt inner product, we have
	\begin{equation}\label{eq:36}
		\begin{split}
			f\square^{\ast}(\xi)=f\langle P,\nabla^{2}\xi\rangle&=f\sum_{i,j}\langle P(e_{i},e_{j}),\nabla^{2}\xi(e_{i},e_{j})\rangle\\
			&=\sum_{i,j}e_{j}\left(f\langle P(e_{i},e_{j}),\nabla^{\perp}_{e_{i}}\xi\rangle\right)-\sum_{i,j}e_{i}\left(e_{j}(f)\langle P(e_{i},e_{j}),\xi\rangle\right)\\
			&\quad-f\sum_{i,j}\langle\nabla^{\perp}_{e_{j}}P(e_{i},e_{j}),\nabla^{\perp}_{e_{i}}\xi\rangle+\sum_{i,j}e_{i}(e_{j}(f))\langle P(e_{i},e_{j}),\xi\rangle\\
			&\qquad\quad+\sum_{i,j}e_{j}(f)\langle\nabla^{\perp}_{e_{i}}P(e_{i},e_{j}),\xi\rangle.
		\end{split}
	\end{equation}
	On the other hand, by a direct computation
	\begin{equation}\label{eq:37}
		\begin{split}
			\sum_{i,j}e_{i}(e_{j}(f))\langle P(e_{i},e_{j}),\xi\rangle&=m\sum_{i,j}e_{i}(e_{j}(f))\delta_{ij}\langle h,\xi\rangle-\sum_{i,j}e_{i}(e_{j}(f))\langle\sigma(e_{i},e_{j}),\xi\rangle\\
			&=m\Delta f\langle h,\xi\rangle-\sum_{\alpha,i,j}e_{i}(e_{j}(f))\langle A_{\alpha}(e_{i}),e_{j}\rangle\langle e_{\alpha},\xi\rangle\\
			&=\sum_{\alpha}\left(mH^{\alpha}\Delta f-{\rm tr}(A_{\alpha}\circ{\rm Hess}\,f)\right)\langle e_{\alpha},\xi\rangle\\
			&=\sum_{\alpha}{\rm tr}(P_{\alpha}\circ{\rm Hess}\,f)\langle e_{\alpha},\xi\rangle=\langle\square(f),\xi\rangle,
		\end{split}
	\end{equation}
	where $\delta_{ij}=\langle e_i,e_j\rangle$. Inserting~\eqref{eq:37} in~\eqref{eq:36} we get
	\begin{equation}\label{eq:38}
		\begin{split}
			f\square^{\ast}(\xi)&=\langle\square(f),\xi\rangle-f\sum_{i,j}\langle\nabla^{\perp}_{e_{j}}P(e_{i},e_{j}),\nabla^{\perp}_{e_{i}}\xi\rangle+\sum_{i,j}e_{j}(f)\langle\nabla^{\perp}_{e_{i}}P(e_{i},e_{j}),\xi\rangle\\
			&\qquad+\sum_{i,j}e_{j}\left(f\langle P(e_{i},e_{j}),\nabla^{\perp}_{e_{i}}\xi\rangle\right)-\sum_{i,j}e_{i}\left(e_{j}(f)\langle P(e_{i},e_{j}),\xi\rangle\right).
		\end{split}
	\end{equation}
	
	We observe that the last expressions in~\eqref{eq:38} can be seen as divergences, that is,
	\begin{equation}\label{eq:39}
		\sum_{i,j}{\rm div}\left(e_{j}(f)\langle P(e_{i},e_{j}),\xi\rangle e_{i}\right)=\sum_{i,j}e_{j}(f)\langle P(e_{i},e_{j}),\xi\rangle{\rm div}(e_{i})+\sum_{i,j}e_{i}\left(e_{j}(f)\langle P(e_{i},e_{j}),\xi\rangle\right)
	\end{equation}
	and
	\begin{equation}\label{eq:40}
		\begin{split}
			\sum_{i,j}{\rm div}\left(f\langle P(e_{i},e_{j}),\nabla^{\perp}_{e_{i}}\xi\rangle e_{j}\right)=f\sum_{i,j}\langle P(e_{i},e_{j}),\nabla^{\perp}_{e_{i}}\xi\rangle{\rm div}(e_{j})+\sum_{i,j}e_{j}\left(f\langle P(e_{i},e_{j}),\nabla^{\perp}_{e_{i}}\xi\rangle\right).
		\end{split}
	\end{equation}
	Since at $p\in\Sigma^{m}$ it holds ${\rm div}(e_{i})(p)=0$ for any $1\leq i \leq m$, we obtain
	\begin{equation}\label{eq:41}
		\begin{split}
			\sum_{i,j}{\rm div}&\left(f\langle P(e_{i},e_{j}),\nabla^{\perp}_{e_{i}}\xi\rangle e_{j}-e_{j}(f)\langle P(e_{i},e_{j}),\xi\rangle e_{i}\right)\\
			&=\sum_{i,j}e_{j}\left(f\langle P(e_{i},e_{j}),\nabla^{\perp}_{e_{i}}\xi\rangle\right)-\sum_{i,j}e_{i}\left(e_{j}(f)\langle P(e_{i},e_{j}),\xi\rangle\right).
		\end{split}
	\end{equation}
	Now, by using the Codazzi equation~\eqref{eq:Codazzi},
	\begin{equation}
		\begin{split}
			\langle\nabla^{\perp}_{e_{j}}P(e_{i},e_{j}),\nabla^{\perp}_{e_{i}}\xi\rangle&=m\delta_{ij}\langle\nabla^{\perp}_{e_{j}}h,\nabla^{\perp}_{e_{i}}\xi\rangle-\langle\nabla^{\perp}_{e_{i}}\sigma(e_{j},e_{j}),\nabla^{\perp}_{e_{i}}\xi\rangle+\langle\overline{R}(e_{i},e_{j})\nabla^{\perp}_{e_{i}}\xi,e_{j}\rangle\\
			&=m\delta_{ij}\langle\nabla^{\perp}_{e_{j}}h,\nabla^{\perp}_{e_{i}}\xi\rangle-\langle\nabla^{\perp}_{e_{i}}\sigma(e_{j},e_{j}),\nabla^{\perp}_{e_{i}}\xi\rangle\\
			&\quad+\langle\nabla^{\perp}_{e_{i}}\xi,N\rangle\left(\langle e_{j},T\rangle\delta_{ij}-\langle e_{i},T\rangle\right),
		\end{split}
	\end{equation}
	and hence
	\begin{equation}\label{eq:42}
		f\sum_{i,j}\langle\nabla^{\perp}_{e_{j}}P(e_{i},e_{j}),\nabla^{\perp}_{e_{i}}\xi\rangle=-(m-1)f\langle\nabla^{\perp}_{T}\xi,N\rangle.
	\end{equation}
	In a similar way,
	\begin{equation}\label{eq:43}
		\sum_{i,j}e_{j}(f)\langle\nabla^{\perp}_{e_{i}}P(e_{i},e_{j}),\xi\rangle=-(m-1)\langle N,\xi\rangle\langle\nabla f,T\rangle.
	\end{equation}
	Replacing~\eqref{eq:41}, \eqref{eq:42} and~\eqref{eq:43} all of these in~\eqref{eq:38},
	\begin{equation}\label{eq:44}
		\begin{split}
			f\square^{\ast}(\xi)&=\langle\square(f),\xi\rangle+\sum_{i,j}{\rm div}\left(f\langle P(e_{i},e_{j}),\nabla^{\perp}_{e_{i}}\xi\rangle e_{j}-e_{j}(f)\langle P(e_{i},e_{j}),\xi\rangle e_{i}\right)\\
			&\qquad+(m-1)\left(f\langle\nabla^{\perp}_{T}\xi,N\rangle-\langle N,\xi\rangle\langle\nabla f,T\rangle\right).
		\end{split}
	\end{equation}
	It is worth pointing our that the expression in the divergence term is independent of the chosen frame. Finally, by using the divergence theorem, we obtain the desired result.
\end{proof}

In particular, taking $f\equiv1$ in Lemma~\ref{lem:1}, we get

\begin{corollary}\label{cor:1}
	Let $\Sigma^{m}$ be a closed submanifold in the product space $\mathbb{S}^{n}\times\mathbb{R}$. Then, for all $\xi\in \mathfrak{X}(\Sigma)^{\perp}$
	\begin{equation}
		\int_{\Sigma}\square^{\ast}(\xi)d\Sigma=(m-1)\int_{\Sigma}\langle\nabla^{\perp}_{T}\xi,N\rangle d\Sigma.
	\end{equation}
\end{corollary}

The next result gives a Hiusken type inequality for submanifolds in $\mathbb{S}^n\times\mathbb{R}$.
\begin{lemma}\label{lem:2}
	If $\Sigma^{m}$ is a submanifold in the product space $\mathbb{S}^{n}\times\mathbb{R}$, then
	\begin{equation}\label{Hiusken-type}
		|\nabla^{\perp}\sigma|^{2}\geq\dfrac{m}{m+2}\left(3m|\nabla^{\perp}h|^{2}+4(m-1)\langle\nabla^{\perp}_{T}h,N\rangle\right).
	\end{equation}
\end{lemma}

\begin{proof}
	Let $F:\mathfrak{X}(\Sigma)^{3}\to\mathfrak{X}(\Sigma)^{\perp}$ be the tensor defined by
	\begin{equation}\label{eq:45}
		F(X,Y,Z)=\nabla_{Z}^{\perp}\sigma(X,Y)+a\left(\langle Y,Z\rangle\nabla_{X}^{\perp}h+\langle X,Z\rangle\nabla_{Y}^{\perp}h+\langle X,Y\rangle\nabla_{Z}^{\perp}h\right),
	\end{equation}
	for a given $a\in\mathbb{R}$. Let us compute its norm. A direct computation gives
	\begin{equation}\label{eq:46}
		\langle F(X,Y,Z),F(X,Y,Z)\rangle=\langle\nabla_{Z}^{\perp}\sigma(X,Y),\nabla_{Z}^{\perp}\sigma(X,Y)\rangle+2aQ_{1}(X,Y,Z)+a^{2}Q_{2}(X,Y,Z), 
	\end{equation}
	where
	\begin{equation}\label{eq:47}
		\begin{split}
			Q_{1}(X,Y,Z)&=\langle Y,Z\rangle\langle\nabla^{\perp}_{X}h,\nabla^{\perp}_{Z}\sigma(X,Y)\rangle+\langle X,Z\rangle\langle\nabla^{\perp}_{Y}h,\nabla^{\perp}_{Z}\sigma(X,Y)\rangle\\
			&\quad+\langle X,Y\rangle\langle\nabla^{\perp}_{Z}h,\nabla^{\perp}_{Z}\sigma(X,Y)\rangle,
		\end{split}
	\end{equation}
	and
	\begin{equation}\label{eq:48}
		\begin{split}
			Q_{2}(X,Y,Z)&=\left(\langle Y,Z\rangle^{2}\langle\nabla^{\perp}_{X}h,\nabla^{\perp}_{X}h\rangle+\langle X,Z\rangle^{2}\langle\nabla^{\perp}_{Y}h,\nabla^{\perp}_{Y}h\rangle+\langle X,Y\rangle^{2}\langle\nabla^{\perp}_{Z}h,\nabla^{\perp}_{Z}h\rangle\right)\\
			&\quad+2\langle Y,Z\rangle\langle X,Z\rangle\langle\nabla^{\perp}_{X}h,\nabla^{\perp}_{Y}h\rangle+2\langle Y,Z\rangle\langle X,Y\rangle\langle\nabla^{\perp}_{X}h,\nabla^{\perp}_{Z}h\rangle\\
			&\quad+2\langle X,Z\rangle\langle X,Y\rangle\langle\nabla^{\perp}_{Y}h,\nabla^{\perp}_{Z}h\rangle.
		\end{split}
	\end{equation}
	
	Given $p\in \Sigma^{m}$, and $\{e_1,\ldots,e_m\}$ an orthonormal frame of $\mathfrak{X}(\Sigma)$ on a neighbourhood $U\subset \Sigma^{m}$ of $p$, which is geodesic at $p$, it is not difficult to check that
	\begin{equation}\label{eq:49}
		\sum_{i,j,k}\langle\nabla_{e_{k}}^{\perp}\sigma(e_{i},e_{j}),\nabla_{e_{k}}^{\perp}\sigma(e_{i},e_{j})\rangle=|\nabla^{\perp}\sigma|^{2}\quad\mbox{and}\quad\sum_{i,j,k}Q_{2}(e_{i},e_{j},e_{k})=3(m+2)|\nabla^{\perp}h|^{2}.
	\end{equation}
	Besides that, from the Codazzi equation~\eqref{eq:Codazzi} we have
	\begin{equation}\label{eq:50}
		\begin{split}
			\sum_{i,j,k}Q_{1}(e_{i},e_{j},e_{k})&=\sum_{i,j,k}\left(\delta_{jk}\langle\nabla^{\perp}_{e_{i}}h,\nabla^{\perp}_{e_{k}}\sigma(e_{i},e_{j})\rangle+\delta_{ik}\langle\nabla^{\perp}_{e_{j}}h,\nabla^{\perp}_{e_{k}}\sigma(e_{i},e_{j})\rangle+\delta_{ij}\langle\nabla^{\perp}_{e_{k}}h,\nabla^{\perp}_{e_{k}}\sigma(e_{i},e_{j})\rangle\right)\\
			&=3m|\nabla^{\perp}h|^{2}+2(m-1)\sum_{i}\langle e_{i},T\rangle\langle\nabla^{\perp}_{e_{i}}h,N\rangle\\
			&=3m|\nabla^{\perp}h|^{2}+2(m-1)\langle\nabla^{\perp}_{T}h,N\rangle.
		\end{split}
	\end{equation}
	Hence,
	\begin{equation}\label{eq:51}
		|F|^{2}=|\nabla^{\perp}\sigma|^{2}+2a\left(3m|\nabla^{\perp}h|^{2}+2(m-1)\langle\nabla^{\perp}_{T}h,N\rangle\right)+3a^{2}(m+2)|\nabla^{\perp}h|^{2}.
	\end{equation}
	Taking $a=-m/(m+2)$ we obtain~\eqref{Hiusken-type}.
\end{proof}

\section{Proof of Theorem~\ref{teo:1}}

From now on, we will deal with $\mathcal{H}$-surfaces immersed in the product space $\mathbb{S}^{n}\times\mathbb{R}$. Before proving our main result, Theorem~\ref{teo:1}, we need the following auxiliary proposition.

\begin{proposition}\label{prop:3}
	Let $\Sigma^{2}$ be an $\mathcal{H}$-surface in the product space $\mathbb{S}^{n}\times\mathbb{R}$. Then, we have
	\begin{equation}\label{eq:prop}
		\int_{\Sigma}\left(|\nabla^{\perp}\sigma|^{2}+2\sum_{\alpha}{\rm tr}(A_{\alpha}\circ{\rm Hess}\,H^{\alpha})\right)d\Sigma\geq\int_{\Sigma}\left(2\langle N,h\rangle^{2}-(2-|T|^{2}+|\phi|^{2})H^{2}\right)d\Sigma.
	\end{equation}
\end{proposition}

\begin{proof}
	Firstly, taking into account the definition of $P$, a direct computation gives us
	\begin{equation}\label{eq:52}
		\begin{split}
			\langle P,\nabla^{2}\xi\rangle&=\sum_{i,j}\langle P(e_{i},e_{j}),\nabla^{2}\xi(e_{i},e_{j})\rangle\\
			&=2\sum_{i,j}\delta_{ij}\langle h,\nabla^{2}\xi(e_{i},e_{j})\rangle-\sum_{i,j}\langle\sigma(e_{i},e_{j}),\nabla^{2}\xi(e_{i},e_{j})\rangle\\
			&=2\langle h,\Delta^{\perp}\xi\rangle-\sum_{i,j}\langle\sigma(e_{i},e_{j}),\nabla^{2}\xi(e_{i},e_{j})\rangle,
		\end{split}
	\end{equation}
	for any orthonormal frame $\{e_1,e_2\}$ of $\mathfrak{X}(\Sigma)$. Furthermore, with a similar reasoning as the one in~\eqref{eq:37}, we get
	\begin{equation}
		\sum_{i,j}\langle\sigma(e_{i},e_{j}),\nabla^{2}\xi(e_{i},e_{j})\rangle=\sum_{\alpha}{\rm tr}(A_{\alpha}\circ{\rm Hess}\,\xi^{\alpha}),
	\end{equation}
	where $\xi^{\alpha}:=\langle\xi,e_{\alpha}\rangle$. Therefore,
	\begin{equation}\label{eq:55}
		\square^{\ast}(\xi)=2\langle h,\Delta^{\perp}\xi\rangle-\sum_{\alpha}{\rm tr}(A_{\alpha}\circ{\rm Hess}\,\xi^{\alpha}).
	\end{equation}
	Making $\xi=2h$ in~\eqref{eq:55}, we write
	\begin{equation}\label{eq:56}
		\square^{\ast}(2h)=4\langle\Delta^{\perp}h,h\rangle-2\sum_{\alpha}{\rm tr}(A_{\alpha}\circ{\rm Hess}\,H^{\alpha})
	\end{equation}
	On the other hand, by using the following identity
	\begin{equation}\label{eq:57}
		\dfrac{1}{2}\Delta H^{2}=\langle\Delta^{\perp}h,h\rangle+|\nabla^{\perp}h|^{2},
	\end{equation}
	\eqref{eq:56} reads
	\begin{equation}\label{eq:58}
		\square^{\ast}(2h)=\langle\Delta^{\perp}h,h\rangle+\dfrac{3}{2}\Delta H^{2}-3|\nabla^{\perp}h|^{2}-2\sum_{\alpha}{\rm tr}(A_{\alpha}\circ{\rm Hess}\,H^{\alpha}).
	\end{equation}
	By using Lemma~\ref{lem:2} in the case $m=2$,
	\begin{equation}\label{eq:59}
		-3|\nabla^{\perp}h|^{2}\geq-|\nabla^{\perp}\sigma|^{2}+2\langle\nabla^{\perp}_{T}h,N\rangle.
	\end{equation}
	Hence,
	\begin{equation}\label{eq:60}
		\square^{\ast}(2h)\geq\langle\Delta^{\perp}h,h\rangle+\dfrac{3}{2}\Delta H^{2}-|\nabla^{\perp}\sigma|^{2}+2\langle\nabla^{\perp}_{T}h,N\rangle-2\sum_{\alpha}{\rm tr}(A_{\alpha}\circ{\rm Hess}\,H^{\alpha}).
	\end{equation}
	
	Let us consider now $\left\{e_3,\ldots,e_{n+1}\right\}$ a normal orthonormal frame in $\mathfrak{X}(\Sigma)^{\perp}$. Then, by writing $h=\sum_{\alpha}H^{\alpha}e_{\alpha}$ and taking into account the definition of $\phi_\alpha$, we easily get
	\begin{equation}\label{eq:27}
		\begin{split}
			\sum_{\alpha,\beta}H^{\alpha}{\rm tr}(A_{\alpha} A_{\beta})\langle e_{\beta},h\rangle&=\sum_{\alpha,\beta,\gamma}H^{\alpha}H^{\gamma}{\rm tr}(A_{\alpha} A_{\beta})\langle e_{\beta},e_{\gamma}\rangle\\
			&=\sum_{\alpha,\beta}H^{\alpha}H^{\beta}{\rm tr}(\phi_{\alpha}\phi_{\beta})+2\sum_{\alpha,\beta}(H^{\alpha})^{2}(H^{\beta})^{2}\\
			&=\sum_{\alpha,\beta}H^{\alpha}H^{\beta}{\rm tr}(\phi_{\alpha}\phi_{\beta})+2H^{4}.
		\end{split}
	\end{equation}
	So, by Proposition~\ref{prop:2},
	\begin{equation}\label{eq:63}
		\langle\Delta^{\perp}h,h\rangle+\left(2-|T|^{2}\right)H^{2}-2\langle N,h\rangle^{2}+\sum_{\alpha,\beta}H^{\alpha}H^{\beta}{\rm tr}(\phi_{\alpha}\phi_{\beta})=0.
	\end{equation}
	Now let us consider $\sigma_{\alpha\beta}={\rm tr}(\phi_{\alpha}\phi_{\beta})$ for all $3\leq \alpha,\beta\leq n+1$. Observe that the $(n-1)\times(n-1)$-matrix $(\sigma_{\alpha\beta})$ is symmetric and it can be assumed to be diagonal for a suitable choice of the normal orthonormal frame $\{e_{3},\ldots,e_{n+1}\}$. Hence,
	\begin{equation}\label{eq:62}
		\sum_{\alpha,\beta}H^{\alpha}H^{\beta}{\rm tr}(\phi_{\alpha}\phi_{\beta})=\sum_{\alpha}(H^{\alpha})^{2}{\rm tr}(\phi_{\alpha}^{2})\leq\sum_{\alpha}(H^{\alpha})^{2}\sum_{\beta}{\rm tr}(\phi_{\beta}^{2})=H^{2}|\phi|^{2}.
	\end{equation}
	Replacing~\eqref{eq:63} and~\eqref{eq:62} in~\eqref{eq:60},
	\begin{equation}\label{eq:64}
		\begin{split}
			\square^{\ast}(2h)-2\langle\nabla^{\perp}_{T}h,N\rangle&\geq-(2-|T|^{2}+|\phi|^{2})H^{2}+2\langle N,h\rangle^{2}+\dfrac{3}{2}\Delta H^{2}\\
			&\quad-|\nabla^{\perp}\sigma|^{2}-2\sum_{\alpha}{\rm tr}(A_{\alpha}\circ{\rm Hess}\,H^{\alpha}).
		\end{split}
	\end{equation}
	Finally, Proposition~\ref{prop:3} is proved taking into account Corollary~\ref{cor:1} and the divergence theorem.
\end{proof}

Now, we are in position to present the proof of Theorem~\ref{teo:1}.

\begin{proof}[\underline{Proof of Theorem~\ref{teo:1}}]
	To begin with, taking into account the definition of $\phi_\alpha$ it is immediate to check that for all $3\leq\alpha,\beta\leq n+1$ it holds
	\begin{equation}\label{eq:66}
		A_{\alpha} A_{\beta}-A_{\beta}A_{\alpha}=\phi_{\alpha}\phi_{\beta}-\phi_{\beta}\phi_{\alpha}.
	\end{equation}
	Furthermore, since for any $3\leq \alpha \leq n+1$ $\phi_\alpha$ is a $2\times 2$ symmetric matrix with ${\rm tr}(\phi_\alpha)=0$, we easily get $\phi_\alpha^2=\lambda I$ for a certain $\lambda\in\mathbb{R}$ and, consequently,
	\begin{equation}\label{eq:trab}
		{\rm tr}(\phi_\alpha^2\phi_\beta)=0
	\end{equation} 
	for all $3\leq \alpha,\beta\leq n+1$.
	
	Besides that, with a straightforward computation and considering~\eqref{eq:trab} we can get the following algebraic identities:
	\begin{equation}\label{eq:65}
		\sum_{\alpha,\beta}{\rm tr}(A_{\beta}){\rm tr}(A^{2}_{\alpha} A_{\beta})=2H^{2}|\phi|^{2}+4H^{4}+4\sum_{\alpha,\beta}H^{\alpha}H^{\beta}{\rm tr}(\phi_{\alpha}\phi_{\beta}),
	\end{equation}
	and
	\begin{equation}\label{eq:67}
		\sum_{\alpha,\beta}[{\rm tr}(A_{\alpha} A_{\beta})]^{2}=\sum_{\alpha,\beta}[{\rm tr}(\phi_{\alpha}\phi_{\beta})]^{2} +4H^{4}+4\sum_{\alpha,\beta}H^{\alpha}H^{\beta}{\rm tr}(\phi_{\alpha}\phi_{\beta}).
	\end{equation}
	Hence, from all the above identities,
	\begin{equation}\label{eq:68}
		\begin{split}
			-\sum_{\alpha,\beta}&\left(N(A_{\alpha} A_{\beta}-A_{\beta} A_{\alpha})+[{\rm tr}(A_{\alpha} A_{\beta})]^{2}-{\rm tr}(A_{\beta}){\rm tr}(A^{2}_{\alpha} A_{\beta})\right)\\
			&=-\sum_{\alpha,\beta}\left(N(\phi_{\alpha}\phi_{\beta}-\phi_{\beta}\phi_{\alpha})+[{\rm tr}(\phi_{\alpha}\phi_{\beta})]^{2}\right)+2H^{2}|\phi|^{2}.
		\end{split}
	\end{equation} 
	So, Proposition~\ref{prop:1} can be written as follow
	\begin{equation}\label{eq:70}
		\begin{split}
			\dfrac{1}{2}\Delta|\sigma|^{2}=&\,\,|\nabla^{\perp}\sigma|^{2}+2\sum_{\alpha}{\rm tr}(A_{\alpha}\circ{\rm Hess}\,H^{\alpha})+2|\phi_{N}|^{2}-4\sum_{\alpha}|\phi_{\alpha}(T)|^{2}\\
			&+\left(2-|T|^{2}+2H^{2}\right)|\phi|^{2}-2\langle\phi_{h}(T),T\rangle\\
			&\qquad-\sum_{\alpha,\beta}\left(N(\phi_{\alpha}\phi_{\beta}-\phi_{\beta}\phi_{\alpha})+[{\rm tr}(\phi_{\alpha}\phi_{\beta})]^{2}\right).
		\end{split}
	\end{equation}
	Observe now that, by using Lemma~\ref{lem:3.4},
	\begin{equation}\label{eq:71}
		\begin{split}
			-\sum_{\alpha,\beta}\left(N(\phi_{\alpha}\phi_{\beta}-\phi_{\beta}\phi_{\alpha})+[{\rm tr}(\phi_{\alpha}\phi_{\beta})]^{2}\right)\geq-\dfrac{3}{2}|\phi|^{4}.
		\end{split}
	\end{equation}
	Moreover, the Cauchy-Schwarz's inequality implies
	\begin{equation}\label{eq:75}
		-4\sum_{\alpha}|\phi_{\alpha}(T)|^{2}\geq-4|\phi|^{2}|T|^{2}\quad\mbox{and}\quad-2\langle\phi_{h}(T),T\rangle\geq-2|\phi_{h}||T|^{2}.
	\end{equation}
	Inserting~\eqref{eq:71} and~\eqref{eq:75} in~\eqref{eq:70} we get
	\begin{equation}\label{eq:72}
		\begin{split}
			\dfrac{1}{2}\Delta|\sigma|^{2}\geq&\,\,|\nabla^{\perp}\sigma|^{2}+2\sum_{\alpha}{\rm tr}(A_{\alpha}\circ{\rm Hess}\,H^{\alpha})+2|\phi_{N}|^{2}-2|\phi_{h}||T|^{2}\\
			&\quad+\left(2-5|T|^{2}+2H^{2}-\dfrac{3}{2}|\phi|^{2}\right)|\phi|^{2}.
		\end{split}
	\end{equation}
	
	Taking integrals and using the divergence theorem, it follows from Proposition~\ref{prop:3} that
	\begin{equation}\label{eq:77}
		\begin{split}
			0\geq\int_{\Sigma}&\left\{2(|\phi_N|^2+\langle N,h\rangle^2)+\left(|T|^2+|\phi|^2\right)H^2\right\}d\Sigma\\
			&+\int_{\Sigma}\left\{\left(2-5|T|^2-\frac{3}{2}|\phi|^2\right)|\phi|^2-2H^2-2|\phi_h||T|^2\right\}d\Sigma.
		\end{split}	
	\end{equation}
	Hence,
	\begin{equation}\label{eq:78}
		\int_{\Sigma}\left\{\left(2-5|T|^2-\frac{3}{2}|\phi|^2\right)|\phi|^2-2H^2-2|\phi_h||T|^2\right\}d\Sigma\leq0.
	\end{equation}
	On the other hand, by the Gauss equation~\eqref{eq:Gauss} it holds
	\begin{equation}\label{eq:79}
		2H^{2}=2K+|\phi|^{2}-2(1-|T|^{2}).
	\end{equation}
	Then, the Gauss-Bonnet theorem implies
	\begin{equation}\label{eq:80}
		\int_{\Sigma}\left\{\left(1-5|T|^2-\dfrac{3}{2}|\phi|^{2}\right)|\phi|^{2}-2(|\phi_{h}|+1)|T|^{2}+2\right\}d\Sigma\leq4\pi\rchi(\Sigma).
	\end{equation}
	
	Finally, let us study when the equality holds in~\eqref{eq:80}. In such case, all the inequalities obtained along the proof should become equalities. In particular, the equality in~\eqref{eq:77} and~\eqref{eq:78} holds. Thus, $|\phi_N|=\langle N,h\rangle=0$ and either $|T|=|\phi|=0$ or $H=0$. In the first case, $\Sigma^{2}$ is an $\mathcal{H}$-surface satisfying the assumptions of Corollary~\ref{cor:tu_tg}, so it is totally geodesic. Therefore, either it is isometric to a slice $\mathbb{S}^2\times\{t_0\}$ in the case $n=2$, or to a totally geodesic sphere $\mathbb{S}^2$ in a certain $\mathbb{S}^3\times\{t_0\}$.
	
	Let us focus in the second case. On the one hand, since $|\phi_{N}|=\langle N,h\rangle=0$,~\eqref{eq:7} implies that $A_{N}=0$. Consequently, from~\eqref{eq:6} we have that $|T|$ is constant on $\Sigma^{2}$, and so it is $|N|$. On the other hand, since $H=0$ and the equality also holds in Lemma~\ref{lem:2}, $\Sigma^{2}$ is necessarily a parallel surface of $\mathbb{S}^{2}\times\mathbb{R}$. Then, the Codazzi equation~\eqref{eq:Codazzi} reads
	\begin{equation}\label{eq:81}
		0=\langle\overline{R}(X,Y)Z,N\rangle=|N|^{2}\left(\langle X,T\rangle\langle Y,Z\rangle-\langle Y,T\rangle\langle X,Z\rangle\right),
	\end{equation}
	for all $X,Y,Z\in\mathfrak{X}(\Sigma)$. Therefore, we easily get that either $N=0$ or $T=0$. In the case where $N=0$, we must have that $\Sigma^{2}$ is a vertical cylinder $\pi^{-1}(\gamma)$, $\gamma$ being a circle in $\mathbb{S}^{2}$ and $\pi:\mathbb{S}^{2}\times\mathbb{R}\rightarrow\mathbb{S}^{2}$ the natural projection map. This case cannot occurs, since it contradicts the compactness assumption of $\Sigma^{2}$. Hence, $T=0$, so $\Sigma^{2}$ is a minimal surface in a slice of $\mathbb{S}^{n}\times\mathbb{R}$. For the case where $\Sigma^{2}$ can be isometrically immersed in a certain $\mathbb{S}^3\times\{t_0\}$, a classical result of isoparametric surfaces in Riemannian space forms~\cite{Lawson:69} guarantees that $\Sigma^{2}$ is isometric to a Clifford torus $\mathbb{S}^{1}(1/\sqrt{2})\times\mathbb{S}^{1}(1/\sqrt{2})$ in $\mathbb{S}^{3}\times\{t_{0}\}$, for some $t_{0}\in\mathbb{R}$. In other case, observe that, again from~\eqref{eq:7}, $|\phi|^2=|\sigma|^2$, so the equality in~\eqref{eq:78} becomes
	\begin{equation}\label{eq:82}
		\int_{\Sigma}|\sigma|^{2}\left(\dfrac{3}{2}|\sigma|^{2}-2\right)d\Sigma=0.
	\end{equation}
	Therefore, from~\cite[Theorem~1]{Li:92} $\Sigma^{2}$ is isometric to a Veronese surface in $\mathbb{S}^{4}\times\{t_{0}\}$, for some $t_{0}\in\mathbb{R}$.
\end{proof}

\section*{Acknowledgements}
The first author is partially supported by the grant PID2021-124157NB-I00, funded by MCIN/ AEI/10.13039/501100011033/ ``ERDF A way of making Europe'' and by the Regional Government of Andalusia FEDER Project 1380930-F. The third author is partially supported by CNPq, Brazil, grants 431976/2018-0 and 311124/2021-6 and Propesqi (UFPE).

\end{document}